\documentclass[11pt,reqno]{amsart}

\usepackage{graphics,graphicx}
\usepackage{amsmath,amsthm}
\usepackage{amssymb,amsfonts}
\usepackage{multirow}
\usepackage{color}
\usepackage{xcolor}
\usepackage{setspace}
\usepackage{enumerate}

\renewcommand{\and}{~~{\rm  and}~~}

\newcommand{\0}{\emptyset}

\newcommand{\tv }{{\rm TV}}

\topmargin by -\headsep \textheight 9.1in \oddsidemargin -7pt
\evensidemargin \oddsidemargin \marginparwidth 0.5in
\newcommand{\bb}{\Big}

\newcommand{\beq}[1]{\begin{equation}\label{#1}}
\newcommand{\mn}[0]{\medskip\noindent}

\newtheorem{thm}{Theorem}[section]
\newtheorem{cor}[thm]{Corollary}
\newtheorem{lem}[thm]{Lemma}
\newtheorem{prop}[thm]{Proposition}

\newcommand{\bea}[1]{\begin{eqnarray}\label{#1}}
\newcommand{\eea}{\end{eqnarray}}

\newcommand{\gd}{\delta}

\newcommand{\bn}{\bigskip\noindent}
\newcommand{\sm}{\setminus}
\newcommand{\ra}{\rightarrow}
\newcommand{\pr}{\Pr}
\newcommand{\sub}{\subseteq}

\newcommand{\eps}{\varepsilon}

\newcommand{\bean}{\begin{eqnarray*}}
\newcommand{\eean}{\end{eqnarray*}}

\newcommand{\bin}{{\rm Bin}}

\newcommand{\eeq}{\end{equation}}
\newcommand{\old}[1]{}
\newcommand{\pp}{\partial}
\newcommand{\lsim}{~\mbox{\raisebox{-.6ex}{$\stackrel{\textstyle{<}}{\sim}$}}~}

\newcommand{\h}{{\mathcal H}}

\renewcommand{\include}{\old}
\newcommand{\edc}{\end{document}}

\newcommand{\ph}{\hat{p}}

\newcommand{\bp}{\mathbf{p}_{_{\! \tiny 4}}}
\renewcommand{\pp}{p_{_{\! 2}}}
\renewcommand{\tt}{t_{_{\! 0}}}
\newcommand{\qq}{q_{_{ 0}}}
\newcommand{\qtwo}{q_{_{ 2}}}

\newcommand{\ppp}{p_{_{\! 3}}}
\newcommand{\pk}{p_{_{\! k}}}
\newcommand{\rk}{r_{_{\! k}}}

\newcommand{\pl}{p_{_{\! 4}}}
\newcommand{\G}{\mathcal{G}}

\def\({\left(}
\def\){\right)} 
\def\[{\left[}
\def\]{\right]} 

\newcommand{\cA}{\mathcal{A}}

\newcommand{\reg}{\topmargin -3pt \advance \topmargin by
-\headheight \advance \topmargin by -\headsep \textheight 9.1in
\oddsidemargin -.1in \evensidemargin \oddsidemargin
\marginparwidth 0.5in \textwidth 6.7in}

\DeclareMathOperator{\var}{Var}

\begin{document}

\title{On the total variation distance between the binomial random graph and the random intersection graph}

\author[J.~H.~Kim]{Jeong Han Kim}
\address{School of Computational Sciences \\ Korea Institute for Advanced Study \\ Seoul, South Korea {\rm(J.~H.~Kim)}} 
\email{jhkim@kias.re.kr}
\thanks{The first author was supported by the National Research Foundation of Korea (NRF) Grant funded by the Korean Government (MSIP) (NRF-2012R1A2A2A01018585) and KIAS internal Research Fund CG046001.}

\author[S.~J.~Lee]{Sang June Lee}
\address{Department of Mathematics \\  Duksung Women's University \\ Seoul, South Korea {\rm(S.~J.~Lee)}} 
\email{sanglee242@duksung.ac.kr, sjlee242@gmail.com}
\thanks{The second author was supported by Basic Science Research Program through the National Research Foundation of Korea (NRF) funded by the Ministry of Science, ICT \& Future Planning (NRF-2013R1A1A1059913). This work was partially done while the second author was visiting Korea Institute for Advanced Study (KIAS)} 

\author[J.~Na]{Joohan Na}
\address{School of Computational Sciences \\ Korea Institute for Advanced Study \\ Seoul, South Korea {\rm(J.~Na)}} 
\email{jhna@kias.re.kr, najoohan@gmail.com}
\thanks{The third author was supported by KIAS internal Research Fund CG053601.}

\date{\today}

\maketitle

\begin{abstract}
When each vertex is assigned a set, the intersection graph generated by the sets is the graph in which two distinct vertices are joined by an edge if and only if their assigned sets have a nonempty intersection. An interval graph is an intersection graph generated by intervals in the real line. A chordal graph can be considered as an intersection graph generated by subtrees of a tree. In 1999, Karo\'nski, Scheinerman and Singer-Cohen [Combin Probab Comput 8 (1999), 131--159] introduced a random intersection graph by taking randomly assigned sets. The random intersection graph $G(n,m;p)$ has $n$ vertices and sets assigned to the vertices are chosen to be i.i.d. random subsets of a fixed set $M$ of size $m$ where each element of $M$ belongs to each random subset with probability $p$, independently of all other elements in $M$. Fill, Scheinerman and Singer-Cohen [Random Struct Algorithms 16 (2000), 156--176] showed that the total variation distance between the random graph $G(n,m;p)$ and the Erd\"os-R\'enyi graph $G(n,\hat{p})$ tends to $0$ for any $0 \leq p=p(n) \leq 1$ if $m=n^{\alpha}$, $\alpha >6$, where $\ph$ is chosen so that the expected numbers of edges in the two graphs are the same. In this paper, it is proved  that the total variation distance still tends to $0$ for any $0 \leq p=p(n) \leq 1$ whenever $m \gg n^4$. 
\end{abstract}

\section{Introduction}
The {\it intersection graph} on  $V:=\{1,\ldots,n\}$ generated by a collection $\{ L_1, \ldots, L_n\}$ of sets is the graph on $V$ in which two distinct vertices $i$ and $j$ are adjacent if and only if their corresponding sets $L_i$ and $L_j$ have a nonempty intersection. In 1945, Szpilrajn-Marczewski \cite{intgraph45} observed that every graph may be represented as an intersection graph. Later, Erd\H{o}s, Goodman and P\'osa \cite{EGP66} showed that every graph with $n$ vertices can be represented as an intersection graph generated by subsets of a set of $n^2/4$ elements. An interval graph is an intersection graph generated by intervals in the real line. A chordal graph turned out to be an intersection graph generated by subtrees of a tree \cite{chordal74}. In general, a class of graphs is called an \textit{intersection class} of a family $\mathcal{F}$ of sets if each graph in the class is an intersection graph generated by sets in $\mathcal{F}$. Scheinerman \cite{Sch85} found a necessary and sufficient condition for a class of graphs to be an intersection class of a family $\mathcal{F}$ of sets. Intersection graphs have been applied to phylogeny problems in biology \cite{biochord10}, seriation problems in psychology \cite{Hub74}, and contingency tables in statistics \cite{Khamis97}, etc. For more details, see \cite{TopIG99}.

In 1999, Karo\'nski, Scheinerman and Singer-Cohen \cite{KSS99} introduced the \textit{random intersection graph}, which is the intersection graph generated by independent and identically distributed (i.i.d.) random subsets $L_1, ..., L_n$ of $M=\{1,..., m\}$. Fill, Scheinerman and Singer-Cohen \cite{FSS00} considered conditions under which the random intersection graph is essentially the binomial random graph (that is, the Erd\H{o}s-R\'enyi random graph with independently chosen edges) with the same expected number of edges. Let $G(n,m;p)$ denote the random intersection graph generated by i.i.d. random subsets $L_1, ..., L_n$ whose distributions are binomial with parameters $(m,p)$, i.e., for a subset $A$ of $M$, $\pr[L_i =A] = p^{|A|}(1-p)^{m-|A|}$. Fill, Scheinerman and Singer-Cohen were interested in how close $G(n,m;p)$ is to $G(n,\ph)$ in terms of total variation distance, where $\ph$ is chosen so that the expected numbers of edges in the two graphs are the same, i.e., 
\begin{equation*}
\ph:=1-(1-p^2)^m.
\end{equation*}

The total variation distance between two (graph-valued) random variables $X$ and $Y$ is defined by 
$$
\tv\(X, Y\)=\dfrac{1}{2}\sum_{G}\bb|\Pr\[X=G\]-\Pr\[Y=G\] \bb|,
$$
where the sum is taken over all possible values of $X$ and $Y$. 

\begin{thm}[{\cite[Theorem 10]{FSS00}}]\label{thm:FSS}
Let $\alpha>6$ be a constant and $m = n^{\alpha}$. Then for any $0 \leq p=p(n) \leq 1$, 
$$ \tv\bb(G(n,m;p), G(n,\ph)\bb)=o(1).$$
\end{thm}
\noindent For $3 < \alpha \leq 6$, Rybarczyk \cite{R_equiv} proved a weaker result. Namely, for any monotone property $\cA$, $\pr[G(n,m;p) \in \cA]$ and $\pr[G(n,\ph) \in \cA]$ are essentially the same. The exact statements of the theorems there are rather complicated.

 A random intersection graph has received a lot of attention due to a great diversity of applications in areas such as epidemics \cite{epidemic08}, circuit design \cite{KSS99}, network user profiling \cite{profile04} and analysis of complex networks \cite{Bloz08, degcluster13, assort13, DK09}. For more information, we refer the reader to the survey papers \cite{Bl1, Bl2, Sp}. For instance, $G(n,m;p)$ is applicable for gate matrix circuit design, which is related to the optimization problem of finding a permutation of the order of gate lines that minimizes the number of horizontal tracks required to lay out the circuit. The problem is NP-hard in general, but it is solvable in $O(n)$ time when $G$ is an interval graph \cite{AGT04}. Karo\'nski, Scheinerman and Singer-Cohen \cite{KSS99} studied conditions for which $G(n,m;p)$ is an interval graph with high probability.

When $L_i$'s are uniformly distributed in the class of subsets of $M$ of the same size, the random intersection graph generated by the $L_i$'s is called a \textit{uniform random intersection graph}. An application to security of wireless sensor networks \cite{connect09, BJR09, Di06, R_diam} is one of the main motivations for studying the uniform random intersection graph. The random intersection graph can be generalized in the way that the vertices $i$ and $j$ are adjacent if $L_i$ and $L_j$ have at least $s \geq 1$ common elements.
The generalization is applicable for cluster analysis \cite{degcluster13, assort13, GJ03}.

The random intersection graph $G(n,m;p)$ may be defined using an $n\times m$ random matrix $R(n,m;p)$ whose rows are indexed by $i\in V$ and columns are indexed by $a\in M$. Each entry of the matrix is $1$ or $0$ with probability of $p$ and $1-p$, respectively, independently of all other entries. The row vector indexed by $ i \in V$ corresponds to the subset $L_i$ of $M$. On the other hand, the column vector indexed by $a\in M$ corresponds to the set $V_a$ of all vertices $i\in V$  with $a\in L_i$. The graph $G(n,m;p)$ may be alternatively constructed by taking the edge set to be the union of edge sets of the complete graphs on $V_a$ for all $a \in M$. 

The main difference between $G(n,m;p)$ and $G(n,\ph)$ are  the complete graphs  induced by the column vectors with three or more $1$'s. In particular, the triangles formed by the columns with exactly three 1's play an important role. Those triangles are to be called {\it artifact triangles}. Roughly speaking, if $mp^{2}$ is large, then $\ph$ is close to $1$ so that both of $G(n,m;p)$ and $G(n,\ph)$ are almost the complete graphs with high probability. On the other hand, if $mp^{2}$ is small, then the expected number of artifact triangles is ${n \choose 3}mp^3(1-p)^{n-3} = O( \frac{n^3}{m^{1/2}} (mp^2)^{3/2})$, which goes to $0$, provided $m \gg n^6$. Theorem \ref{thm:FSS} was proved based upon this observation. 

In this paper, we will show that the total variation distance is still small enough even if there are some artifact triangles. It is actually small as long as the expected number of pairs of distinct artifact triangles with a common edge is small. If the expected number is not small, the total variation distance may be small when both of $G(n,m;p)$ and $G(n,\ph)$ are almost the complete graphs with high probability. Based on these two facts, we infer that if $m \gg n^4$ then the total variation distance is always small for any $p$: It turns out that the expected number is $O(n^4 m^2 p^6)$. To have the total variation distance small for all $p$, it is required that $mp^2$ is large when $n^4 m^2 p^6=\frac{n^4}{m} (mp^2)^3$ is not small, which holds if $m \gg n^4$.

\mn

\begin{thm}\label{main1}
For $m \gg n^4$ and $0\leq p =p(n)\leq 1$, we have that
$$\tv\bb(G(n,m;p), G(n,\ph)\bb)=o(1).$$
\end{thm}

\mn

In the next section, we give the outline of the proof of Theorem \ref{main1}. The proof will be divided into four parts, which will be proved in Sections \ref{sec:main_proof}-\ref{sec:ineq3}.

\section{Preliminaries and Outline of proof of Theorem \ref{main1}}\label{sec:main_proof}

If $ p \geq \bb(\frac{3\log n}{m}\bb)^{1/2}$, both of $G(n,m;p)$ and $ G(n,\ph)$ are the complete graphs with probability $1-O\(\frac{1}{n}\)$. Indeed, for each edge $e$, 
$$ 
\pr [ e \not \in  G(n,m ; p ) ]  = (1-p^2)^m  \leq e^{-mp^2}\leq \frac{1}{n^3},
$$
and hence $G(n,m;p)$ is the complete graph with probability $1-O\(\frac{1}{n}\)$. Since the expected numbers of edges in $G(n,m;p)$ and $G(n,\ph)$ are the same, $G(n,\ph)$ is the complete graph with probability $1-O\(\frac{1}{n}\)$ as well. Therefore,
$$\tv\bb( G(n,m;p), G(n,\ph)\bb) = O\bb(\frac{1}{n}\bb).$$ 

In the rest of the paper, we assume that
$$
0 \leq  p \leq \bb(\frac{3\log n}{m}\bb)^{1/2}.
$$
As described in the introduction, the random intersection graph $G(n,m;p)$ may be constructed using an $n \times m$ random matrix $R(n,m;p)$ whose rows are indexed by $v \in V$ and columns are indexed by $a \in M$. For fixed $a \in M$, the probability of $V_a :=\{v \in V : a \in L_v \}$ being a fixed $k$-subset of $V$ is $p^k (1-p)^{n-k}$ for integer $k\geq 2$. Hence $V_a$ is the $k$-subset for some $a \in M$ with probability $1-(1- p^k (1-p)^{n-k})^m$, which will be approximated by 
$$ p_{_{\! k}}:= 1-e^{-mp^k (1-p)^{n-k}}.$$
Also, $G(n,m; p)$ will be approximated by another random graph $G(n, (\pk) )$,  which is to be defined below. 

For  $0\leq p^* \leq 1$, let $\h_k (n,p^*)$ be a random collection of $k$-subsets of  $V$ to  which  each $k$-subset belongs with probability $p^*$, independently of all other $k$-subsets. For $H\sub V$, let $K(H)$ be the complete graph on $H$. Then, for a collection $\h$ of subsets of $V$, let $K(\h)$ denote  the graph on $V$ whose edge set is the union of edge sets of the complete graphs $K(H)$ on  $H\in \h$. Notice that $K(\h_2(n,p^*))$ is the binomial random graph $G(n,p^*)$. For $\pk$ defined above, let $G(n, (\pk) )$ be the random graph on $V$ whose edge set is the union of edge sets of $K(\h_2(n,\pp)), K(\h_3(n,\ppp)), \ldots, K(\h_k(n,\pk)), \ldots$.

For $m\gg n^4$ and  $p \leq \(\frac{3\log n}{m}\)^{1/2}$, the probability of $\displaystyle \bigcup_{k\geq 5}\h_k(n,\pk)$ being nonempty is upper bounded by 
\begin{equation*}
\sum_{k\geq 5}{n \choose k } \pk \leq \sum_{k\geq 5} n^k mp^k = O\bb( \frac{n^5\log^3 n}{m^{3/2}} \bb)
=O\bb(\frac{\log^3 n}{n} \bb).
\end{equation*}
Thus, for $G(n, \pp, \ppp, \pl ) = G(n, (\pp, \ppp, \pl,$
$0, \ldots) )$, 
$$\tv\bb(G(n,(\pk)), G(n,\pp,\ppp,\pl)\bb) \leq \pr\bb[ \bigcup_{k\geq 5}\h_k(n,\pk) \neq \0\bb] = O\bb(\frac{\log^3 n}{n} \bb).$$
We will further approximate $G(n,\pp, \ppp, \pl)$ by $G(n, \pp)$, which is the main contribution of this paper. 

\mn

Summarizing all, since the total variation distance between $G(n,m;p)$ and $G(n,\ph)$ is upper bounded by the sum of $\tv(G(n,m;p), G(n,(\pk)))$, $\tv(G(n,(\pk)), G(n, \pp, \ppp, \pl))$, $\tv(G(n,\pp,\ppp,\pl), G(n,\pp))$ and $\tv(G(n,\pp), G(n,\ph))$, it is enough to show that each total variation distance tends to $0$.  For the second one is $O\bb(\frac{\log^3 n}{n} \bb)$ described as above, we will prove that the other three total variation distances tend to $0$ in Sections \ref{sec:ineq1}, \ref{sec:ineq2} and \ref{sec:ineq3}, respectively.

\section{total variation distance between $G(n,m;p)$ and $G(n, (\pk) )$}\label{sec:ineq1}

To prove that the total variation distance between $G(n,m;p)$ and $G(n, (\pk) )$ tends to $0$, we will use a coupling argument. For two random variables $X$ and $Y$, a coupling $(X',Y')$ of $X$ and $Y$ is a vector of random variables such that the marginal distributions of $(X',Y')$ are the distributions of $X$ and $Y$, respectively. The total variation distance between $X$ and $Y$ is upper bounded by the probability of $X' \neq Y'$ for any coupling $(X',Y')$ of $X$ and $Y$. On the other hand, there always exists a coupling $(X',Y')$ so that the total variation distance of $X$ and $Y$ is equal to the probability of $X' \neq Y'$.

\begin{lem}{\cite[Chapter I, Theorem 5.2]{coupling}}
Let $X$ and $Y$ be random variables. Then any coupling $(X',Y')$ of $X$ and $Y$ satisfies
$$\tv(X,Y) \leq \Pr[X' \neq Y'].$$
Moreover, there exists a coupling for which the equality holds, i.e.,
$$\tv(X,Y) = \Pr[X' \neq Y'].$$
\end{lem}

Using an appropriate coupling between a binomial random variable and a Poisson random variable, we will prove the following proposition, which may be applied for the case $m \gg n^2 \log n$. The proposition is essentially the same as Lemma 5 in \cite{R_equiv}. We prove it for the sake of completeness.

\begin{prop}
Let $m\gg n^2 \log n $, $0\leq p \leq  (\frac{3\log n}{m})^{1/2}$ and 
$p_{_{\! k}}= 1-e^{-mp^k (1-p)^{n-k}}$ for integers $k\geq 2$. Then
$$\tv\bb(G(n,m;p), G(n,(\pk))\bb) = O\bb( \frac{n^2 \log n}{m} \bb).$$ 
\end{prop}

\begin{proof}
Let $X$ be the number of columns of the matrix $R(n,m;p)$ with two or more 1's, or equivalently, the number of $a\in M$  with  $|V_a| \geq 2$. Since 
$$ \pr [ |V_a| =k ] = {n \choose k} p^k (1-p)^{n-k} =: \rk $$ 
for any fixed $a\in M$, the random variable $X$ has the binomial distribution with parameters $m$ and $ \qtwo :=\sum_{k\geq 2} \rk$, i.e.,
$$ \pr [ X= \ell ] ={ m \choose  \ell  } (\qtwo)^{\ell} (1-\qtwo )^{m-\ell}. $$

The random graph $G(n,m;p)$ may be constructed as follows: Take i.i.d. random complete graphs $K^{(1)}, ..., K^{(h)}$, ... on subsets of $V$, where the number of vertices in  $K^{(1)}$ is $k\geq 2$ with probability $\rk/\qtwo$, and then, once the number is given to be $k$, every $k$-subset of $V$ is equally likely to be the vertex set of $K^{(1)}$. In other words, for a $k$-subset $U$ of $V$ with $k\geq 2$, the probability of $U$ being the vertex set of $K^{(1)}$ is $\frac{\rk}{\qtwo} { n \choose k}^{-1}.$ (As $\sum_{k\geq 2} \frac{\rk}{\qtwo} =1$, the random complete graph $K^{(1)}$ is well-defined.) The edge set of $G(n,m;p)$ is the union of edge sets of $X$ random complete graphs $K^{(1)}, ..., K^{(X)}$.

\mn

We now take a Poisson random variable $Y$ with mean $m\qtwo$ that is coupled with $X$ so that $\pr[X \neq Y] = \tv(X,Y)$. Let $G_Y$ be the graph whose edge set is the union of edge sets of $K^{(1)}, ..., K^{(Y)}$. Then
$$ \tv(G(n,m;p), G_Y) \leq \pr[G(n,m;p) \neq G_Y] \leq \pr[X \neq Y] = \tv(X,Y).$$

On the other hand, $G_Y$ has the same distribution as $G(n, (\pk))$. Indeed, for each subset $U$ of $V$ with $|U|\geq 2$, let  $Z(U) $ be the number of $i=1,2,..., Y$ such that the vertex set of $K^{(i)}$ is $U$. Then, it is well-known that for $k=|U|$, $Z(U) $'s are independent Poisson random variables with mean $m \qtwo \cdot \frac{\rk}{\qtwo} {n \choose k}^{-1}=m p^k(1-p)^{n-k}$, and hence $\pr[ Z(U) >0 ] = 1- e^{-m p^k(1-p)^{n-k}} =\pk $. Since the edge set of $G_Y$ is the union of edge sets of the  complete graphs on $U$ with $ Z(U)>0$, $G_Y$ has the same distribution as $G(n, (\pk))$.

The desired bound follows from the fact that the total variation distance between the binomial random variable $X$ with parameters $m,\qtwo$ and the Poisson random variable $Y$ with mean $m\qtwo$ is not more than $\qtwo$ \cite[Theorem 2.4]{poi89}, and

$$\qtwo = \sum_{k \geq 2}{n \choose k} p^k (1-p)^{n-k} \leq  \sum_{k \geq 2} n^k p^k = O\big(n^2p^2\big) = O\bb( \frac{n^2 \log n}{m} \bb).$$ 
\end{proof}

\section{total variation distance between $G(n, \pp, \ppp, \pl )$ and $G(n, \pp )$}\label{sec:ineq2}

In this section, we prove that the total variation distance between $G(n, \pp, \ppp, \pl )$ and $G(n, \pp )$ tends to $0$. This is the main contribution of the paper. Intuitively, if there are no artifact triangles (and no columns with at least four 1's) with high probability, then $G(n, \pp, \ppp, \pl )$ and $G(n, \pp )$ should be almost the same. We will show that $\tv( G(n, \pp, \ppp, \pl ), G(n, \pp ) )$ is still small enough even if there are few artifact triangles. 
As mentioned earlier, it actually turns out that the distance is small enough if the expected number of pairs of distinct artifact triangles with
a common edge is small. When the expected number  is not small, the total variation distance tends to $0$ provided that $mp^2$ is sufficiently large.
 Keeping this in mind, we prove the following proposition.

\begin{prop}
Let  $m\gg n^4 $, $0 \leq p \leq  (\frac{3\log n}{m})^{1/2}$ and $\pk= 1-e^{-mp^k (1-p)^{n-k}}$ for $k \geq 2$. Then
$$\tv\bb( G(n, \pp, \ppp, \pl ), G(n, \pp ) \bb) = O(\eps), $$
where 
\begin{equation}\label{eq:eps}\eps:= \max\bb\{ \frac{1}{\log n}, \frac{1}{\log (m/n^4)}\bb\}.\end{equation}
\end{prop}
For simplicity, we write $G(n,\bp)$ for $G(n, \pp, \ppp, \pl )$. It is not difficult to check that
\begin{equation}\label{eq:tv}
\tv\bb(G(n,\bp), G(n,\pp)\bb) = \sum_{G \in \G} \Big( \Pr[G(n,\pp)=G] -\min\big\{\Pr[G(n,\bp)=G], \Pr[G(n,\pp)=G]\big\}\Big),
\end{equation}
where $\G$ is the set of all graphs on $V$. In order to bound the total variation distance, we consider a lower bound of $\Pr[G(n, \bp)=G]$. Since $G(n, \bp) = K(\h_4 (n,\pl))\cup K(\h_3(n,\ppp))\cup G(n,\pp),$ we may write $\pr[ G(n, \bp) = G ]$ as the sum of 
\begin{equation}\label{eq:4,3,2} 
\pr\bb[\h_4 (n,\pl) =Q, \hskip 0.5em \h_3 (n,\ppp) =T, \hskip 0.5em G\sm(K(T) \cup K(Q)) \sub G(n, \pp) \sub G \bb] 
\end{equation}
over all possible $T$ and $Q$. 
Let $\h_3(G)$ and $\h_4(G)$ be the collections of all $K_3$'s and $K_4$'s in $G$ that are regarded as collections of $3$-subsets and  $4$-subsets of $V$, respectively. Then, 
\begin{equation*}
\begin{split}
\pr[ G(n, \bp) = G ] & = \sum_{Q \sub \h_4 (G) \atop T \sub \h_3(G)} \pl^{|Q|} (1-\pl)^{{n \choose 4}-|Q|} \ppp^{|T|} (1-\ppp)^{{n \choose 3}-|T|} \pp^{|G|-|K(Q)\cup K(T)|} (1-\pp)^{{n \choose 2}-|G|} \\
& = \pr[ G(n,\pp) =G] \sum_{Q \sub \h_4 (G) \atop T \sub \h_3(G)} \pl^{|Q|} (1-\pl)^{{n \choose 4}-|Q|} \ppp^{|T|} (1-\ppp)^{{n \choose 3}-|T|} \pp^{-|K(Q)\cup K(T)|},
\end{split}
\end{equation*}
where $|G|$ is  the number of edges in $G$. Let $G\setminus K(Q)$ be the graph obtained from $G$ by removing the edges of the graph $K(Q)$. For each $Q\sub \h_4 (G)$, taking only the case that $T \sub \h_3(G\setminus K(Q))$ yields that  
\begin{equation}\label{main}
\frac{\pr[ G(n, \bp) = G ]}{\pr[ G(n,\pp) =G]} \geq  \sum_{Q \sub \h_4 (G) } \pl^{|Q|} (1-\pl)^{{n \choose 4}-|Q|} \pp^{-|K(Q)|}
\sum_{T \sub \h_3(G\sm K(Q))} \ppp^{|T|} (1-\ppp)^{{n \choose 3}-|T|}
\pp^{-|K(T)|}. 
\end{equation}

\mn

In the case that the expected number ${n \choose 3}\ppp = \Theta(n^3mp^3)$ of artifact triangles is small, say $p \leq \frac{\eps}{nm^{1/3}}$, one may take $T, Q=\0$ in the lower bound of \eqref{main} to obtain
$$
\pr[ G(n, \bp) = G ]\geq \pr[ G(n,\pp) =G] (1-\pl)^{{n \choose 4}}(1-\ppp)^{{n \choose 3}},
$$
and then \eqref{eq:tv} gives that
\begin{equation*}
\begin{split}
\tv\bb(G(n,\bp), G(n,\pp)\bb) & \leq \sum_{G \in \G} \Pr[G(n,\pp)=G] \( 1 - (1-\pl)^{{n \choose 4}}(1-\ppp)^{{n \choose 3}}\) = O(\eps)
\end{split}
\end{equation*}
as  $ {n \choose 3}\ppp   = \Theta (n^3 m p^3)=O(\eps)$ and ${n \choose 4} \pl   =\Theta(n^4 m  p^4)=O(\eps)$. If $m =n^{\alpha}$ for $\alpha >6$, then this holds for all $p \leq \bb(\frac{3 \log n}{m}\bb)^{1/2}$ since $\frac{\eps}{nm^{1/3}} \geq \bb(\frac{3\log n}{m}\bb)^{1/2}$, which essentially implies the result of \cite{FSS00}.

\bn

We now assume that
$$
\frac{\eps}{nm^{1/3}} < p \leq \bb(\frac{3 \log n}{m}\bb)^{1/2}.
$$
For any set $\G^*$ of graphs on V, using \eqref{eq:tv}, we have that the total variation
distance is at most  
$$
\pr[G(n,\pp) \notin \G^*] + 
\sum_{G \in \G^*} \Big( \Pr[G(n,\pp)=G] -\min\big\{\Pr[G(n,\bp)=G], \Pr[G(n,\pp)=G]\big\}\Big).
$$
Therefore it
should be enough to consider the graphs $G$ satisfying 
$$
|\h_3(G)| \approx {n \choose 3}\pp^3 \and |\h_4(G)| \approx {n \choose 4}\pp^6,
$$
the exact meaning of which will be defined later.

We first give an intuition behind the  proof that will be given later. Recalling~\eqref{main}, it turns out that

\begin{align}
\sum_{T \sub \h_3(G\sm K(Q))} \ppp^{|T|} (1-\ppp)^{{n \choose 3}-|T|}
\pp^{-|K(T)|} & \leq \sum_{t\geq 0} \sum_{T \sub \h_3(G) \atop |T|=t } \ppp^{t} (1-\ppp)^{{n \choose 3}-t} \pp^{-|K(T)|} \nonumber \\
& \lsim \sum_{t\geq 0} {{n \choose 3}\pp^3 \choose t} \ppp^{t} (1-\ppp)^{{n \choose 3}-t}
\pp^{-3t}. \label{eq:asymp_upper}
\end{align}
Since $\displaystyle {{n \choose 3}\pp^3 \choose t} \leq {{n \choose 3} \choose t}\pp^{3t}$, it follows that
$$
\sum_{t\geq 0} {{n \choose 3}\pp^3 \choose t} \ppp^{t} (1-\ppp)^{{n \choose 3}-t}
\pp^{-3t} \leq \sum_{t\geq 0} {{n \choose 3} \choose t} \ppp^{t} (1-\ppp)^{{n \choose 3}-t} =1.
$$
Similarly,
$$
\sum_{Q \sub \h_4 (G) } \pl^{|Q|} (1-\pl)^{{n \choose 4}-|Q|} \pp^{-|K(Q)|} \lsim 1.
$$
Therefore, the lower bound of \eqref{main} is close to $1$ only when all the upper bounds are quite tight. In particular, to have the inequality \eqref{eq:asymp_upper} tight, we need that $|K(T)| = 3t$ for most collections $T$ of $t$ triangles in $G$  for $t \approx {n \choose 3}\ppp$ unless $\pp$ is almost $1$. If $t$ is not close to ${n \choose 3}\ppp$, then the summands are small enough to be negligible. Note that $|K(T)| = 3t$ means that there is no pair of triangles in $\h_3(n,\ppp) = T$ with a common edge.
We consider two cases below depending upon whether the expected number of pairs of artifact triangles $\Theta({n \choose 4}{m \choose 2} p^6) = \Theta(n^4m^2p^6)$ is small or not. 

We will prove the following two lemmas, from which the main proposition easily follows. Recall that $ \eps= \max\bb\{ \frac{1}{\log n}, \frac{1}{\log (m/n^4)}\bb\}$ and $\pk = 1-e^{-mp^k(1-p)^{n-k}}$.

\mn

\begin{lem}\label{lem:caseI}
Suppose that
$$ m \gg n^4 \hskip 1em \mbox{and} \hskip 1em
\frac{\eps}{nm^{1/3}} < p \leq \frac{\eps}{n^{2/3} m^{1/3}}.
$$
Then
$$\tv\bb( G(n, \pp, \ppp, \pl ), G(n, \pp ) \bb) = O(\eps).$$
\end{lem}

\begin{lem}\label{lem:caseII}
Suppose that
$$ m \gg n^4 \hskip 1em \mbox{and} \hskip 1em
\frac{\eps}{n^{2/3} m^{1/3}} < p \leq \bb(\frac{3 \log n}{m}\bb)^{1/2}.
$$
Then
$$\tv\bb( G(n, \pp, \ppp, \pl ), G(n, \pp ) \bb) = O(\eps).$$
(If $m /n^4$ is too large, e.g., $m=n^5$, then there is no such $p$, so the conclusion is trivially true. On the other hand, if it is not too large, e.g., $m= n^4 \log \log n$, then there are  $p$ satisfying the conditions.)
\end{lem}

\mn

Before we prove Lemmas \ref{lem:caseI} and \ref{lem:caseII}, three preliminary lemmas are introduced.

\begin{lem}\label{lem:h3G}
For $m \gg n^4$ and $\frac{\eps}{nm^{1/3}} < p \leq \frac{\eps}{n^{2/3} m^{1/3}}$, suppose that a graph $G$ on $V$ satisfies
\begin{enumerate}[(i)]
\item $|\h_3(G)| \geq (1-\gd){n \choose 3}\pp^3$, where $\gd := \frac{1}{\eps} \( \frac{1-\pp}{n^2\pp}+ \frac{1-\pp}{n^3 \pp^3} \)^{1/2}$,
\item the number $I(G)$ of diamond graphs (i.e., $K_4$ minus one edge) in G  is at most $n^4 \pp^5/\eps$.
\end{enumerate}
Then the number of sets $T$ such that $T \sub \h_3(G)$,  $|T|=t$ and  $ |K(T)|=3t $ is at least 
\begin{equation*} 
(1-O(\eps)) {{n \choose 3} \choose t} \pp^{3t} \qquad \mbox{ for }~0\leq t\leq \tt:= \frac{n^3mp^3}{\eps},
\end{equation*}
where the constant in $O(\eps)$ is independent of $G$ and $t$.
\end{lem}

\begin{proof}
Let $X_t(G)$ be the number of sets $T$ such that $T \sub \h_3(G) $, $|T|=t$ and  $|K(T)|=3t$. We infer that
\begin{equation*}
\begin{split}
X_t(G)&\geq {|\h_3(G)| \choose t}   - I(G){|\h_3(G)|\choose t-2} \\
&=\( 1- \frac{ t(t-1)I(G)}{(|\h_3(G)|-t+2)(|\h_3(G)|-t+1)}\) {|\h_3(G)| \choose t} \\
&\geq \( 1- \frac{ \tt^2I(G) }{(|\h_3(G)|-\tt)^2}\) {|\h_3(G)| \choose t}.
\end{split}
\end{equation*}
Since $|\h_3(G)|=\Omega(n^3p_2^3)$
, we have that
\begin{align}\label{eq:t2/h3G} 
\frac{\tt^2}{|\h_3(G)|}=O\bb(\frac{n^3m^2p^6}{\eps^2p_2^3}\bb)=O\bb(\frac{n^3}{\eps^2m}+\frac{n^3m^2p^6}{\eps^2}\bb)=O\bb(\frac{\eps^2}{n}\bb),
\end{align}
where the second equality follows from $\pp=\Theta\bb(\frac{mp^2}{1+mp^2}\bb)$ and the third equality follows from $p \leq \frac{\eps}{n^{2/3} m^{1/3}}$. In particular $|\h_3(G)|\gg t_0$, and hence
\begin{equation*}
X_t(G)\geq \bb( 1- \frac{ 2\tt^2 I(G)}{|\h_3(G)|^2}\bb) {|\h_3(G)| \choose t}.
\end{equation*}
It is easy to check from \eqref{eq:t2/h3G} that
$$
\frac{2\tt^2 I(G)}{|\h_3(G)|^2}=\frac{2I(G)}{|\h_3(G)|}\cdot\frac{\tt^2}{|\h_3(G)|}
=O\bb(\frac{n^4 \pp^5}{\eps n^3 \pp^3}\cdot \frac{\eps^2}{n}  \bb)
=O(\eps)
$$
and
\begin{align*}
{|\h_3(G)| \choose t} & \geq \bb(1-\frac{\tt}{|\h_3(G)|}\bb)^{\tt} \frac{|\h_3(G)|^t}{t!} \geq \bb(1-O\bb(\frac{\eps^2}{n}\bb)\bb)\frac{|\h_3(G)|^t}{t!} \geq    (1-O(\eps)) (1-\gd)^{\tt} 
 {  {n \choose 3}  \choose t} \pp^{3t}
\end{align*}
as $|\h_3(G)| \geq (1-\gd) {n \choose 3} \pp^3$. Since $\pp=\Theta\bb(\frac{mp^2}{1+mp^2}\bb)$  and 
$p \leq \frac{\eps}{n^{2/3} m^{1/3}}$ yield
\begin{equation}\label{eq:dt} 
( \gd\tt )^2 = O\bb(\frac{n^4mp^4}{\eps^4} + \frac{n^3}{\eps^4 m} \bb) = O\bb( 
\bb(\frac{n^4}{m}\bb)^{1/3} + \frac{1}{n} \bb(\frac{ n^4}{\eps^{4} m} \bb) \bb) 
= O(\eps^2),
\end{equation}
the desired lower bound for $X_t(G)$ follows. 
\end{proof}

\mn

The same argument gives the next lemma regarding $\h_4(G)$.

\begin{lem}\label{lem:h4G}
For $m \gg n^4$ and $\frac{\eps}{n^{2/3}m^{1/3}} < p \leq \(\frac{3 \log n}{m}\)^{1/2}$, suppose that a graph $G$ on $V$ satisfies $$|\h_4(G)| \geq \(1-\frac{1}{\eps n}\){n \choose 4}\pp^6.$$ Then the number of $Q \sub \h_4(G)$ with $|Q|=q$ and  $ |K(Q)|=6q $ is at least 
\begin{equation*} 
(1-O(\eps)) {{n \choose 4} \choose q} \pp^{6q} \qquad \mbox{ for }~ 0\leq q\leq \qq:= \frac{n^4mp^4}{\eps},
\end{equation*}
where the constant in $O(\eps)$ is independent of $G$ and $t$.
\end{lem}

\noindent \textbf{Remark.} The expected number of columns of the matrix $R(n,m;p)$ with four or more 1's is $\Theta(n^4mp^4)$. The parameter $\qq$ is chosen to be substantially, but not extremely, bigger than the expected number $\Theta(n^4mp^4)$.

\begin{proof}
Let $Y_q(G)$ be the number of $Q \sub \h_4(G)$ with $|Q|=q$ and  $ |K(Q)|=6q$. Observe that the number of pairs of $K_4$ in the complete graph on $V$ sharing at least an edge is at most ${n \choose 4} {n \choose 2} {4 \choose 2} \leq n^6$. Thus
\begin{align*}
Y_q(G)&\geq { |\h_4(G)| \choose q } - n^6  { |\h_4(G)| \choose q-2 } \\
& = \bb( 1- \frac{q(q-1)n^6}{(|\h_4(G)|-q+2)(|\h_4(G)|-q+1)} \bb) {|\h_4(G)| \choose q } \\
&\geq  \bb( 1- \frac{q_0^2n^6}{(|\h_4(G)|-q_0)^2} \bb) {|\h_4(G)| \choose q }.
\end{align*}
Since $mp^2 \geq m \( \frac{\eps}{n^{2/3}m^{1/3}}\)^2 = \eps^2 \(\frac{m}{n^4} \)^{1/3} \ra \infty$, we have that $\pp=1-o(1)$ and
$$
|\h_4(G)| \geq \bb(1- \frac{1}{\eps n}\bb){n \choose 4}\pp^6 =(1-o(1)){n \choose 4}.$$ 
Therefore $\qq = \frac{n^4mp^4}{\eps} = O(\eps \log^2 n)$ implies that
$$ \frac{\qq^2n^6}{(|\h_4(G)|-\qq)^2}=O(\eps),$$
and hence
$$
Y_q(G) \geq (1-O(\eps)) {|\h_4(G)| \choose q }.
$$
Since $\frac{\qq^2}{|\h_4(G)|} = O\(\frac{\eps^2 \log^4 n}{n^4}\)= O(\eps)$ and 
$\frac{\qq}{\eps n} = O \(\frac{\log^2 n}{n} \) = O(\eps)$, we have that
\begin{align*}
\binom{|\h_4(G)|}{q} \geq \bb(1-\frac{\qq}{|\h_4(G)|}\bb)^{\qq}\frac{|\h_4(G)|^q}{q!} \geq (1-O(\eps))\bb(1-\frac{1}{\eps n}\bb)^{\qq}\binom{\binom{n}{4}}{q}\pp^{6q} \geq (1-O(\eps))\binom{\binom{n}{4}}{q}\pp^{6q},
\end{align*}
which gives the desired lower bound for $Y_q(G)$.
\end{proof}

\begin{lem}\label{lem:G3G4}
For $\gd = \frac{1}{\eps} \( \frac{1-\pp}{n^2\pp}+ \frac{1-\pp}{n^3 \pp^3} \)^{1/2}$, let $\G_3$ be the set of all graphs $G$ on $V$ satisfying 
$$|\h_3(G)| \geq (1-\gd){n \choose 3}\pp^3 \hskip 0.5em \mbox{and} \hskip 0.5em I(G) \leq n^4 \pp^5/\eps,$$ recalling that $I(G)$ denotes the number of diamond graphs as in Lemma~\ref{lem:h3G},
and let $\G_4$ be the set of all graphs $G$ in $\G_3$ satisfying $$|\h_4(G)| \geq \(1-\frac{1}{\eps n}\){n \choose 4}\pp^6.$$ Then for $m\gg n^4$ we have
$$\pr[G(n,\pp) \in \G_3] = 1-O(\eps) \qquad \mbox{ for }~\frac{\eps}{nm^{1/3}} < p \leq \bb(\frac{3 \log n}{m}\bb)^{1/2}$$
and
$$\pr[G(n,\pp) \in \G_4] = 1-O(\eps) \qquad \mbox{ for }~\frac{\eps}{n^{2/3}m^{1/3}} < p \leq \bb(\frac{3 \log n}{m}\bb)^{1/2},$$
where the constants in $O(\eps)$ are independent of $p$.
\end{lem}

\begin{proof}
For $X_3: =|\h_3(G(n,\pp))|$, Chebyshev's inequality gives that
\begin{align*}
\pr \bb[ X_3 < (1-\delta) {n \choose 3}\pp^3 \bb] &\leq \pr \bb[ |X_3 - E[X_3]| > \delta {n \choose 3}\pp^3 \bb]  \leq \frac{\var[X_3]}{\delta^2 {n \choose 3}^2 \pp^6} =O(\eps^2)
\end{align*}
as $E[X_3] = {n \choose 3}\pp^3$ and $\var[X_3] = O \left( (n^4 \pp^5 + n^3 \pp^3) (1-\pp) \right)$. Moreover, Markov's inequality implies that
$$\pr \bb[ I(G(n,\pp)) > \frac{n^4 \pp^5}{\eps} \bb] \leq \eps$$
since $E\[I(G(n,\pp))\] = {n \choose 4}6\cdot \pp^5 \leq n^4 \pp^5$. Therefore,
$$
\pr[G(n,\pp) \notin \G_3] \leq \pr \bb[ X_3 < (1-\delta) {n \choose 3}\pp^3 \bb] + \pr \bb[ I(G(n,\pp)) > \frac{n^4 \pp^5}{\eps} \bb] = O(\eps).
$$
Similarly, for $X_4 = |\h_4(G(n,\pp))|$, it is not hard to see that 
$$E[X_4] = {n \choose 4}\pp^6 \and \var[X_4] = O \left(  n^6\right)$$
as $\pp=1-o(1)$ for $p > \frac{\eps}{n^{2/3}m^{1/3}}$, and Chebyshev's inequality yields that
\begin{align*}
\pr \bb[ X_4 < \bb(1-\frac{1}{\eps n}\bb){n \choose 4}\pp^6 \bb] &\leq \pr \bb[ |X_4 - E[X_4] | > \frac{1}{\eps n} {n \choose 4}  \pp^6 \bb]  \leq \frac{\eps^2 n^2\var [X_4]}{{n \choose 4}^2 \pp^{12}} =O(\eps^2).
\end{align*}
Therefore,
\begin{equation*}
\pr[G(n,\pp) \notin \G_4] \leq \pr[G(n,\pp) \notin \G_3] + \pr \bb[ X_4 < \bb(1-\frac{1}{\eps n}\bb){n \choose 4}\pp^6 \bb] = O(\eps).
\end{equation*}
\end{proof}

\bn

Now we prove the main lemmas.

\begin{proof}[\textbf{Proof of Lemma \ref{lem:caseI}}]
 Equality \eqref{eq:tv} and Lemma \ref{lem:G3G4} imply that the total variation distance between $G(n,\bp)$ and $G(n,\pp)$ is at most
\begin{equation}\label{eq:G3upper}
\begin{split}
& \pr[G(n,\pp) \notin \G_3] +\sum_{G \in \G_3} \Big( \Pr[G(n,\pp)=G] -\min\big\{\Pr[G(n,\bp)=G], \Pr[G(n,\pp)=G]\big\}\Big) \\
& =  O(\eps) + \sum_{G \in \G_3} \Big( \Pr[G(n,\pp)=G] -\min\big\{\Pr[G(n,\bp)=G], \Pr[G(n,\pp)=G]\big\}\Big).
\end{split}
\end{equation}
Taking $Q=\0$ in~\eqref{main}, we have that
\begin{equation*} 
\begin{split}
\pr[ G(n, \bp) = G ] &\geq \pr[ G(n,\pp) =G] (1-p_4)^{n\choose 4} \sum_{ T \sub \h_3(G)  } \ppp^{|T|} (1-\ppp)^{{n \choose 3}-|T|} \pp^{-|K(T)|} \\
& = (1-O(\eps)) \pr[ G(n,\pp) =G] \sum_{ T \sub \h_3(G)  } \ppp^{|T|} (1-\ppp)^{{n \choose 3}-|T|}
\pp^{-|K(T)|}
\end{split}
\end{equation*}
as ${n \choose 4} \pl  = \Theta (n^4 m p^4) = O(\eps)$.
For $G \in \G_3$, Lemma \ref{lem:h3G} gives that
\begin{equation*}
\begin{split}
\sum_{ T \sub \h_3(G)  } \ppp^{|T|} (1-\ppp)^{{n \choose 3}-|T|} \pp^{-|K(T)|} 
& \geq \sum_{t=0}^{\tt} \sum_{T \sub \h_3(G) \atop |T|=t, |K(T)|=3t  } \ppp^{t} (1-\ppp)^{{n \choose 3}-t} \pp^{-3t} \\ 
&\geq (1-O(\eps))\sum_{t=0}^{\tt}    {{n \choose 3} \choose t}  \ppp^{t} (1-\ppp)^{{n \choose 3}-t},
\end{split}
\end{equation*}
and
\begin{equation*}
\frac{\pr[G(n,\bp)=G]}{\pr[G(n,\pp)=G]} \geq (1-O(\eps)) \sum_{t=0}^{\tt}    {{n \choose 3} \choose t}  \ppp^{t} (1-\ppp)^{{n \choose 3}-t}.
\end{equation*}
Since $\tt = n^3mp^3/\eps = \Theta(n^3\ppp/\eps)$, Markov's inequality yields that
$$
\sum_{t =0}^{\tt} {{n \choose 3} \choose t}  \ppp^{t} (1-\ppp)^{{n \choose 3}-t} = 1 - \pr \bb[ \bin\bb({n \choose 3}, \ppp\bb) > \tt \bb] = 1-O(\eps),
$$
where $\bin(n', p')$ is the binomial random variable with parameters $n'$ and $p'$. Therefore, $$\mbox{$\pr[G(n,\bp)=G] \geq (1-O(\eps)) \pr[G(n,\pp)=G]$ \hskip 1em for $G \in \G_3$,}$$ which together with \eqref{eq:G3upper} implies that
$
\tv\bb(G(n,\bp),G(n,\pp)\bb) = O(\eps)$,  provided 
$$ m \gg n^4 \hskip 1em \mbox{and} \hskip 1em
\frac{\eps}{n^{2/3} m^{1/3}} < p \leq \bb(\frac{3 \log n}{m}\bb)^{1/2}.
$$
\end{proof}

\bn 

\begin{proof}[\textbf{Proof of Lemma \ref{lem:caseII}}]

As in the proof of Lemma \ref{lem:caseI}, it follows from \eqref{eq:tv} and Lemma \ref{lem:G3G4} that
\begin{eqnarray}\label{eq:prlem43_1}
&&\tv\bb(G(n,\bp),G(n,\pp)\bb)\nonumber \\ &\leq& O(\eps) + \sum_{G \in \G_4} \Big( \Pr[G(n,\pp)=G] -\min\big\{\Pr[G(n,\bp)=G], \Pr[G(n,\pp)=G]\big\}\Big).
\end{eqnarray}
Let $Q\sub \h_4 (G)$, and we write $G\sm Q$ for $G \sm K(Q)$ for brevity. For $G \in \G_4$, the sum in the lower bound of \eqref{main} restricted to the cases $|T| \leq \tt= n^3mp^3/\eps$ and $|Q| \leq \qq=n^4mp^4/\eps$ gives
\begin{equation*}
\begin{split}
\frac{\pr[ G(n, \bp) = G ]}{\pr[ G(n,\pp) =G]}
 \geq \sum_{q=0}^{\qq} & \sum_{Q \sub \h_4 (G) \atop |Q|=q, |K(Q)|=6q }
  \pl^{q} (1-\pl)^{{n \choose 4}-q} \pp^{-6q}\sum_{t=0}^{\tt} \sum_{T \sub \h_3(G\sm Q) \atop |T|=t} \ppp^{t} (1-\ppp)^{{n \choose 3}-t}
\pp^{-|K(T)|}.
\end{split}
\end{equation*}
Lemma \ref{lem:h4G} and Markov's inequality imply that 
\begin{equation*}
\begin{split}
\sum_{q=0}^{\qq} \sum_{Q \sub \h_4 (G) \atop |Q|=q, |K(Q)|=6q }
 \pl^{q} (1-\pl)^{{n \choose 4}-q} \pp^{-6q} &\geq  (1-O(\eps)) \sum_{q=0}^{\qq} 
{{n \choose 4} \choose q} \pl^{q} (1-\pl)^{{n \choose 4}-q} \\
& = (1-O(\eps)) \bb(1- \pr \bb[ \bin\bb({n \choose 4}, \pl \bb) > \qq \bb]\bb) \\
& = 1-O(\eps),
\end{split}
\end{equation*}
where $\bin(n', p')$ is a binomial random variable with parameters $n'$ and $p'$. Therefore,
\begin{equation}\label{eq:lowbound2}
\begin{split}
\frac{\pr[ G(n, \bp) = G ]}{\pr[ G(n,\pp) =G]} \geq (1-O(\eps)) \min_{Q \sub \h_4(G) \atop |Q| \leq \qq} \sum_{t=0}^{\tt} \sum_{T \sub \h_3(G\sm Q) \atop |T|=t} \ppp^{t} (1-\ppp)^{{n \choose 3}-t} \pp^{-|K(T)|}.
\end{split}
\end{equation}

For $T \sub \h_3(G)$, let $I^*(T)$ be the number of pairs of distinct triangles in $T$ with a common edge. (It is a bit different from the definition $I(\cdot)$ in Lemma~\ref{lem:h3G}.) For an edge $e$, let $d_T(e)$ be the number of triangles in $T$ which contain $e$. Then 
$$3|T| - |K(T)| 
= \sum_{e:  d_T(e) \geq 2} ( d_T (e) -1) 
\leq \sum_{e: d_T(e) \geq 2} { d_T(e) \choose 2}
= I^*(T).$$
For a fixed $Q \sub \h_4(G)$ with $|Q| \leq \qq$, we will show that the number of $T\sub \h_3(G\sm Q) $ with $|T|=t\leq \tt = n^3mp^3/\eps$ and  $I^*(T) \leq r:=n^4m^2p^6/\eps^3$ is at least
\begin{equation}\label{eq:it}
(1-O( \eps))  {{n \choose 3} \choose t} 
\pp^{3t}.
\end{equation}
Then
\begin{equation*}
\begin{split}
\sum_{t=0}^{\tt} \sum_{T \sub \h_3(G\sm Q) \atop |T|=t } \ppp^{t} (1-\ppp)^{{n \choose 3}-t}
\pp^{-|K(T)|} &\geq \sum_{t=0}^{\tt} \sum_{T \sub \h_3(G\sm Q) \atop |T|=t, I^*(T)\leq r } \ppp^{t} (1-\ppp)^{{n \choose 3}-t} \pp^{-|K(T)|} \\
&\geq (1-O(\eps))\pp^r\cdot \sum_{t=0}^{\tt} {{n \choose 3} \choose t}  \ppp^{t} (1-\ppp)^{{n \choose 3}-t} \\
&\geq (1-O(\eps))\pp^r,
\end{split}
\end{equation*}
where the last inequality follows from Markov's inequality. Since 
$\pp^r = (1- e^{-mp^2(1-p)^{n-2}})^r\geq 1-O(re^{-mp^2})$
and 
$$ re^{-mp^2}= \frac{n^4m^2p^6}{\eps^3}e^{-mp^2} = \frac{n^4}{\eps^3m}\cdot (mp^2)^3 e^{-mp^2} = O(\eps),$$
we have that $\pp^r = 1-O(\eps)$ and
$$
\sum_{t=0}^{\tt} \sum_{T \sub \h_3(G\sm Q) \atop |T|=t } \ppp^{t} (1-\ppp)^{{n \choose 3}-t}
\pp^{-|K(T)|} \geq 1-O(\eps).
$$
This together with \eqref{eq:lowbound2} and \eqref{eq:prlem43_1} completes the proof of Lemma \ref{lem:caseII}.

It remains to prove \eqref{eq:it}. For $t\leq t_0$, we take the uniform random collection $R=R(t)$ of triangles that is equally likely to be $T$ for every $T \sub \h_3(G\sm Q) $ with $|T| =t$. In other words, for every $T \sub \h_3(G\sm Q) $ with $|T| =t$,
$$ \pr  [ R=T ] = {|\h_3(G\sm Q)| \choose t}^{-1}. $$
Since the number of sets $T\sub \h_3(G\sm Q)$ with $|T|=t$ containing a diamond graph is less than or equal to $ I(G) {|\h_3(G\sm Q)| \choose t-2}$, we have that
$$
E[I^*(R)] \leq I(G) {|\h_3(G\sm Q)| \choose t-2}{|\h_3(G\sm Q)| \choose t}^{-1}\leq \frac{I(G) t_0^2}{\(|\h_3(G\sm Q)|-t_0\)^2},
$$ where $I(G)$ is defined in Lemma~\ref{lem:h3G}.
For $G\in \G_4$, since $K(Q)$ has at most $6 |Q| \leq 6 \qq = \frac{6n^4mp^4}{\eps} = O(\eps\log^2 n) $ edges and each edge in $G$ is contained in at most $n$ triangles in $\h_3(G)$,
\begin{equation}\label{eq:h3g-q}
|\h_3 (G\sm Q)| \geq |\h_3 (G)| - 6 \qq n =  |\h_3 (G)| - O(\eps n\log^2 n)= \Theta (n^3).
\end{equation}
As $\tt = \frac{n^3mp^3}{\eps} \ll n^3$ and $I(G) \leq I(K_n) =  6{n \choose 4} \leq n^4$,
$$
E[I^*(R)] = O \bb( \frac{I(G) t_0^2}{n^6} \bb) =O \bb( \frac{\tt^2}{n^2}\bb)=O\bb(\frac{n^4m^2p^6}{\eps^2}\bb)
$$
and Markov's inequality gives that
$$ 
\pr \[ I^*(R) > r\] \leq \frac{\eps^3 E[I^*(R)]}{n^4m^2p^6} = O(\eps).
$$
The number $Z$ of $T\sub \h_3(G\sm Q) $ with $|T|=t$ and  $I^*(T) \leq r$ satisfies 
\begin{equation*}
Z = (1-O( \eps))  {|\h_3(G\sm Q)| \choose t}.
\end{equation*}

Now we estimate ${|\h_3(G\sm Q)| \choose t}$.
Since $G \in \G_4$ and $\pp=1-o(1)$,  it is obtained similarly to \eqref{eq:h3g-q} that
$$ 
|\h_3(G\sm Q)| \geq |\h_3 (G)| - 6\qq n= 
\(1-\gd - O\(\frac{\qq}{n^2}\)\) {n \choose 3} \pp^3,$$
and then
\begin{align*}
 {|\h_3(G\sm Q)| \choose t} &\geq  \bb(1-\gd - O\bb(\frac{\qq}{n^2}\bb)\bb)^{\tt} \bb(1- O\bb( \frac{\tt}{n^3} \bb)\bb)^{\tt} \frac{{n \choose 3}^t}{t!}  \pp^{3t} \\
 &\geq  \bb(1-\tt\gd - O\bb(\frac{\tt \qq}{n^2}\bb)\bb)\bb(1- O\bb( \frac{\tt^2}{n^3} \bb)\bb)
{{n \choose 3} \choose t} \pp^{3t}.
\end{align*}
As in \eqref{eq:dt}, $\gd \tt = O(\eps)$, and it is easy to check that
$$
\frac{\tt \qq}{n^2}= \frac{n^5m^2p^7}{\eps^2} =O\bb( \frac{n^5 \log^{7/2}n}{\eps^2 m^{3/2}}  \bb)=O(\eps) \quad \and \quad \frac{t_0^2}{n^3}= \frac{n^3m^2p^6}{\eps^2}= O\bb(\frac{n^3 \log^3n}{\eps^2m} \bb) =O(\eps).
$$
Therefore, we have that
$$
Z = (1-O( \eps))  {|\h_3(G\sm Q)| \choose t} \geq (1-O(\eps)){{n \choose 3} \choose t} \pp^{3t}.
$$
This completes the proof of \eqref{eq:it}.
\end{proof}

\section{total variation distance between $G(n, \pp )$ and $G(n, \ph )$}\label{sec:ineq3}

For the random graphs $G(n,m;p), G(n,(\pk)), G(n,\pp,\ppp,\pl)$ and $G(n,\pp)$, we have so far considered the total variation distance between the consecutive pairs of them. Finally, a good upper bound for  the total variation distance between $G(n,\pp)$ and $G(n,\ph)$ easily follows from an upper bound for the total variation distance between 
 two binomial distributions $\bin(N,p)$ and $\bin(N,q)$. As a corollary of  Theorem 2.2
in \cite{J2010},
we may have

\begin{cor}\label{lem:sec5}
Let $N$ be a positive integer, and $p$ and $q$ be real numbers satisfying $0< p < q < 1$. For
$\gd$ satisfying $(q-p)N = \gd \sqrt{p(1-p)N}$, i.e., $\gd = (q-p) \sqrt{\frac{N}{p(1-p)}}$, we have
\begin{equation*}
\tv \bb(\bin(N,p),\bin(N,q)\bb) \leq \gd + 3\gd^2.
\end{equation*}
\end{cor}

\bn

\noindent 
Recalling $\pp = 1- e^{-mp^2(1-p)^{n-2}}$, $\ph = 1- (1-p^2)^m$ and $ \eps = \max\bb\{ \frac{1}{\log n}, \frac{1}{\log (m/n^4)}\bb\}$, we have the last inequality  needed. 

\begin{cor}
Suppose that $m \gg n^4$ and $p \leq \(\frac{3\log n}{m}\)^{1/2}$. Then
$$
\tv\bb(G(n,\pp),G(n,\ph)\bb) = O (\eps).
$$
\end{cor}

\begin{proof}
Let $p= \sqrt{\frac{c}{m}}$ for $0 < c \leq 3\log n$. Since $\pp = \Theta\(\frac{mp^2}{1+mp^2} \) = \Theta\(\frac{c}{1+c} \)$ and
\begin{equation*}
\ph  - \pp = e^{-mp^2(1-p)^{n-2}} - e^{m\log(1-p^2)} \leq e^{-mp^2}(e^{nmp^3} - e^{-mp^4}) = O( nmp^3 e^{-mp^2}),
\end{equation*}
we have that
\begin{equation*}
(\ph - \pp)\sqrt{\frac{{n \choose 2}}{\pp(1-\pp)} } = O\bb(\frac{n^2mp^3}{e^{mp^2}\sqrt{\pp(1-\pp)}} \bb) = O \bb( \bb(\frac{n^4}{m}\bb)^{1/2} \frac{c(1+c)}{e^c} \bb) = O (\eps).
\end{equation*}
Therefore, Corollary \ref{lem:sec5} implies that
$$
\tv\bb(G(n,\pp),G(n,\ph)\bb)\leq \tv\bb(\bin \bb({n \choose 2}, \pp \bb), \bin \bb({n \choose 2}, \ph \bb) \bb) = O(\eps).
$$

\end{proof}

\section{Concluding remark}

Fill, Scheinerman and Singer-Cohen \cite{FSS00} showed that the total variation distance between $G(n,m;p)$ and $G(n,\ph)$ tends to $0$ for $m=n^\alpha, \alpha>6$. In this paper, we improve the result. Namely, the total variation distance still goes to $0$ for $m\gg n^4$. If $m \gg n^4$ then the expected number of pairs of artifact triangles with a common edge is small enough, or both of the two random graphs are complete graphs with high probability. This is the main ingredient of the proof of Theorem \ref{main1}.

Our result naturally gives rise to the question whether the condition $m \gg n^4$ is tight. We initially believed that the total variation distance between $G(n,m;p)$ and $G(n,\ph)$ is not close to $0$ if $m$ is smaller than $n^4$. However, the more we try  to prove it, the more we feel that our initial belief is baseless. It would not be extremely surprising even if  the total variation distance tends to $0$ for some $m$ much less than $n^4$.

\bibliographystyle{abbrv}
%\bibliography{RIG}

\edc